%% file: 32.tex
\documentclass[12pt]{article}
\usepackage{graphicx}
\usepackage{amsmath}
\usepackage{amsfonts}
\usepackage{amsthm}
\usepackage[T1]{fontenc}
\usepackage{url}
\usepackage{color}
\usepackage[margin=1in]{geometry}
\input{mymacros}

\begin{document}
\title{On the fractional Korn inequality in bounded domains: Counterexamples to the case $ps<1$}
\author{Davit Harutyunyan\thanks{University of California Santa Barbara, harutyunyan@math.ucsb.edu} and 
Hayk Mikayelyan\footnote{University of Nottingham Ningbo China, Hayk.Mikayelyan@nottingham.edu.cn}}
\date{}

\maketitle

\begin{abstract}

The validity of Korn's first inequality in the fractional setting in bounded domains has been open. We resolve this problem by proving that in fact Korn's first inequality holds in the case $ps>1$ for fractional $W^{s,p}_0(\Omega)$ Sobolev fields in open and bounded $C^{1}$-regular domains $\Omega\subset \mathbb R^n$. Also, in the case $ps<1,$ for any open bounded $C^1$ domain $\Omega\subset \mathbb R^n$ we construct counterexamples to the inequality, i.e., Korn's first inequality fails to hold in bounded domains. The proof of the inequality in the case $ps>1$ follows a standard compactness approach adopted in the classical case, combined with a Hardy inequality, and a recently proven Korn second inequality by Mengesha and Scott [\textit{Commun. Math. Sci.,} Vol. 20, N0. 2, 405--423, 2022]. The counterexamples constructed in the case $ps<1$ are interpolations of a constant affine rigid motion inside the domain away from the boundary, and of the zero field close to the boundary.
\end{abstract}

\section{Introduction}
\label{sec:1}

The classical Korn's first inequality [\ref{bib:Korn.1},\ref{bib:Korn.2}] states that for any vector field $\Bu\in W^{1,p}_0(\Omega,\mathbb R^n)$ one has 
\begin{equation}
\label{1.1}
\|\nabla\Bu\|_{L^p(\Omega)}\leq C\|e(\Bu)\|_{L^p(\Omega)},
\end{equation}
where $\Omega\subset\mathbb R^n$ is an open bounded Lipschitz domain, $p\in (1,\infty),$ $C>0$ is an absolute constant, and $e(\Bu)=\frac{1}{2}(\nabla\Bu+\nabla\Bu^T)$ is the symmetric part of the gradient. It is also know that if for a field 
$\Bu\in W^{1,p}(\Omega,\mathbb R^n)$ one has $e(\Bu)=0$ a.e. in $\Omega,$ then $\Bu$ must have a constant skew-symmetric gradient 
[\ref{bib:Korn.1},\ref{bib:Korn.2},\ref{bib:Kon.Ole.}]; that is $\Bu(\Bx)=\BA \Bx+\Bb$ for some $\BA\in\mathbb R^{n\times n}$ with $\BA+\BA^T=0,$ and some constant vector $\Bb\in\mathbb R^n.$ The fractional analogue of Korn's first inequality can be formulated as follows. Assuming $s\in (0,1),$ for a vector field 
$\Bu\in L^p(\Omega,\mathbb R^n),$ the seminorm associated to the fractional 
$s-$derivative of $\Bu$ is given by 
\begin{equation}
\label{1.2}
[\Bu]_{W^{s,p}(\Omega)}^p=\int_{\Omega}\int_{\Omega}\frac{|\Bu(\Bx)-\Bu(\By)|^p}{|\Bx-\By|^{n+ps}}d\Bx d\By, 
\end{equation}
while the quantity 
$\|\Bu\|_{W^{s,p}(\Omega)}=[\Bu]_{W^{s,p}(\Omega)}+\|\Bu\|_{L^p(\Omega)}$ defines a norm in the space 
$$W^{s,p}(\Omega,\mathbb R^n)=\{\Bu\in L^p(\Omega,\mathbb R^n) \ : \ \|\Bu\|_{W^{s,p}(\Omega)}<\infty \},$$
that is a Banach space [\ref{bib:Adams}], see also [\ref{bib:Bou.Bre.Mir.1},\ref{bib:Maz.Sha.}]. It then turns out that the fractional analogue of the symmetric part of the $s-$gradient of a field $\Bu\in L^{p}(\Omega,\mathbb R^n)$ is given by (following the notation in [\ref{bib:Mengesha.1}]) 
\begin{equation}
\label{1.3}
[\Bu]_{\mathcal{X}^{s,p}(\Omega)}^p=\int_{\Omega}\int_{\Omega}\frac{|(\Bu(\Bx)-\Bu(\By))\cdot(\Bx-\By)|^p}{|\Bx-\By|^{n+ps+p}}d\Bx d\By.
\end{equation}
It is known that if for a vector field $\Bu\in L^p(\Omega,\mathbb R^n)$ one has $[\Bu]_{\mathcal{X}^{s,p}(\Omega)}=0,$ which is equivalent to $(\Bu(\Bx)-\Bu(\By))\cdot(\Bx-\By)=0$ for a.e. $\Bx,\By\in \Omega,$ then in fact again $\Bu$ has to be an affine map with a skew-symmetric gradient [\ref{bib:Tem.Mir.}, Proposition 1.2], see also [\ref{bib:Mengesha.1}, Theorem~3,1]. This observation suggests, as also noted by Mengesha [\ref{bib:Mengesha.2}] and Mengesha and Scott 
[\ref{bib:Men.Sco.1}], that the fractional analogue of Korn's first inequality should hold as well:
\begin{equation}
\label{1.4}
\int_{\Omega}\int_{\Omega}\frac{|\Bu(\Bx)-\Bu(\By)|^p}{|\Bx-\By|^{n+ps}}d\Bx d\By\leq 
C \int_{\Omega}\int_{\Omega}\frac{|(\Bu(\Bx)-\Bu(\By))\cdot(\Bx-\By)|^p}{|\Bx-\By|^{n+ps+p}}d\Bx d\By,
\end{equation}
for all vector fields $\Bu\in W^{s,p}_0(\Omega,\mathbb R^n),$ where the space $W^{s,p}_0(\Omega,\mathbb R^n)$ is the closure of $C_c^1(\Omega,\mathbb R^n)$ fields in the $W^{s,p}(\Omega)$ norm. Here the constant $C$ in (\ref{1.4}) depends only on $n,p,s$ and $\Omega.$ 
The expression in (\ref{1.3}) has arisen in the theory of linear peridynamics [\ref{bib:Silling.1},\ref{bib:Silling.2},\ref{bib:Sil.Ept.Wec.Xu.Ask.},\ref{bib:Du.Zho.},\ref{bib:Du.Men.},\ref{bib:Du.Lip.}] as the energy in the small strain regime in the analogy of linear elasticity, where the symmetric gradient $e(\Bu)$ is the linear strain, and the integral $\int_\Omega|e(\Bu)|^pd\Bx$ is equivalent to the linear elastic energy. For the purpose of establishing the coercivity of the linear elastic energy, Korn 
proved [\ref{bib:Korn.1},\ref{bib:Korn.2}] the estimate (\ref{1.1}) and the Korn second inequality. The estimate (\ref{1.4}) was first proven by Mengesha in the case when $\Omega$ is the entire space or the upper half-space $\{x_n>0\},$ and $p=2, ps\neq 1$ in [\ref{bib:Mengesha.2}], and was later extended to any
 $1<p<\infty, ps\neq 1$ by Mengesha and Scott in [\ref{bib:Men.Sco.1}]. When dealing with non-local operators like on both sides of (\ref{1.4}), it is sometimes favorable to extend the field $\Bu$ into the entire space $\mathbb R^n$ and work with the extended domain to gain access to Fourier (or other transformation) analysis tools. This strategy has been successfully employed in all of the works [\ref{bib:Mengesha.2},\ref{bib:Men.Sco.1},\ref{bib:Men.Sco.2}], where the idea of Nitsche [\ref{bib:Nitsche}] of extending the field $\Bu$ into $\mathbb R^n$ so that, the symmetric gradient of the extension is suitably controllable plays a vital role. The question of validity of the inequality (\ref{1.4}) for vector fields $\Bu\in W^{s,p}_0(\Omega,\mathbb R^n)$ in bounded domains $\Omega$ was addressed in [\ref{bib:Mengesha.2},\ref{bib:Men.Sco.1},\ref{bib:Men.Sco.2}], and has remained open, and is the topic of the present manuscript. We answer this question by proving that in fact (\ref{1.4}) holds in bounded $C^1$-regular domains in the case $ps>1,$ and fails to hold when $ps<1.$ The interesting fact is that in the case $ps<1,$ it does not only fail in general, but it necessarily fails in any open bounded $C^1$ domains, see Theorem~2.1. The outcome in the case $ps<1$ in bounded domains is in stark contrast with the case of unbounded domains like the entire space or epigraphs, where (1.4) holds as long as $ps\neq 1$, [\ref{bib:Mengesha.2},\ref{bib:Men.Sco.1},\ref{bib:Men.Sco.2}]. However, it is not surprising as Korn's first inequality does not hold in the borderline case $p=1$ in the classical case as shown by the celebrated work of Ornstein [\ref{bib:Ornstein}], see also [\ref{bib:Con.Far.Mag.}] for another approach to constructing counterexamples to such inequalities. For the proof in the case $ps>1,$ following Kondratiev and Oleinik  [\ref{bib:Kon.Ole.}] as for the classical case, we employ the recently proven Korn second inequality by Mengesha and Scott [\ref{bib:Men.Sco.2}, Theorem~1.1] in bounded $C^1$ domains, which reads as follows:
\begin{equation}
\label{1.5}
\int_{\Omega}\int_{\Omega}\frac{|\Bu(\Bx)-\Bu(\By)|^p}{|\Bx-\By|^{n+ps}}d\Bx d\By\leq 
C \left(\int_{\Omega}\int_{\Omega}\frac{|(\Bu(\Bx)-\Bu(\By))\cdot(\Bx-\By)|^p}{|\Bx-\By|^{n+ps+p}}d\Bx d\By+\int_\Omega |\Bu(\Bx)|^pd\Bx\right),
\end{equation}
for all vector fields $\Bu\in W_0^{s,p}(\Omega,\mathbb R^n).$ A considerably shorter proof of (\ref{1.5}) for bounded $C^1$ domains or for bounded Lipschitz 
domains with small Lipschitz constant has been recently given by Rutkowski in [\ref{bib:Rutkowski}, Theorem~1.1]

For the case $ps<1,$ we construct the counterexamples $\Bu\in C_c^1(\Omega,\mathbb R^n)$ by interpolating between affine and zero maps, where the nonzero affine value is taken in most of the interior of the domain $\Omega,$ while the zero values are taken near the boundary. 

Also, an important consequence of the validity of (\ref{1.4}) in the case $ps>1$ is that if one extends any given field $\Bu\in W_0^{s,p}(\Omega,\mathbb R^n)$ onto $\mathbb R^n$ by zero outside $\Omega,$ then the $W^{s,p}(\mathbb R^n)$ and $\mathcal{X}^{s,p}(\mathbb R^n)$ norms of the extended field $\bar\Bu$ remain equivalent to the same norms of the original field $\Bu$ in $\Omega,$ see Theorem~\ref{th:2.2}. 

\section{Main Results}
\setcounter{equation}{0}
\label{sec:2}

The below theorems contain the main results of the paper. 
\begin{theorem}
\label{th:2.1}
Let $s\in (0,1)$ and $p\in [1,\infty)$ such that $ps\neq 1.$ Assume $\Omega\subset\mathbb R^n$ is an open bounded $C^1-$regular domain. Then the following holds:\\
\textbf{(i) The case $ps>1$}. There exists a constant $C=C(n,p,s,\Omega)$ such that 
\begin{equation}
\label{2.1}
\|\Bu\|_{W^{s,p}(\Omega)}\leq C[\Bu]_{\mathcal{X}^{s,p}(\Omega)},
\end{equation}
for all vector fields $\Bu\in W_0^{s,p}(\Omega,\mathbb R^n).$\\ 
\textbf{(ii) The case $ps<1$}. There exists a sequence of vector fields 
$\Bu_k\in C_c^1(\Omega,\mathbb R^n)$ such that
\begin{equation}
\label{2.2}
[\Bu_k]_{\mathcal{X}^{s,p}(\Omega)}\left(\frac{1}{[\Bu_k]_{W^{s,p}(\Omega)}}+\frac{1}{\|\Bu_k\|_{L^{p}(\Omega)}}\right)\to 0
\qquad \text{as}\qquad k\to \infty.
\end{equation}
Consequently, in the case $ps<1,$ Korn's inequality fails to hold in any open bounded $C^1$ domain.

\end{theorem}

The next theorem establishes that if one extends any Sobolev field $\Bu\in W_0^{s,p}(\Omega,\mathbb R^n)$ onto $\mathbb R^n$ by zero outside $\Omega,$ then the seminorms (in this case those are also norms) $[\bar \Bu]_{W^{s,p}(\mathbb R^n)}$ and $[\Bu]_{W^{s,p}(\Omega)}$ as well as the seminorms $[\bar\Bu]_{\mathcal{X}^{s,p}(\mathbb R^n)}$ and $[\Bu]_{\mathcal{X}^{s,p}(\Omega)}$ will still be equivalent in the case $ps>1.$ This was noted by Rutkowski in [\ref{bib:Rutkowski}] under the availability of (\ref{2.1}). Now, keeping in mind that obviously the norms $\|\bar\Bu\|_{L^{p}(\mathbb R^n)}$ and $\|\Bu\|_{L^{p}(\Omega)}$ are always equivalent, this result would allow one to work with fields defined over 
$\mathbb R^n$ instead of $\Omega,$ opening up access to Fourier type analysis in particular.

\begin{theorem}
\label{th:2.2}
Let $s\in (0,1)$ and $p\in (1,\infty)$ such that $ps>1.$ Assume $\Omega\subset\mathbb R^n$ is an open bounded $C^1-$regular domain. 
For any Sobolev field $\Bu\in W_0^{s,p}(\Omega,\mathbb R^n),$ denote by $\bar \Bu$ the extension of $\Bu$ onto $\mathbb R^n$ by zero outside $\Omega.$ There exists a constant $C=C(n,p,s,\Omega)$ such that 
\begin{equation}
\label{2.3}
[\bar \Bu]_{W^{s,p}(\mathbb R^n)}\leq C[\Bu]_{W^{s,p}(\Omega)},
\end{equation}
and
\begin{equation}
[\bar\Bu]_{\mathcal{X}^{s,p}(\mathbb R^n)}\leq C[\Bu]_{\mathcal{X}^{s,p}(\Omega)},
\end{equation}
for all vector fields $\Bu\in W_0^{s,p}(\Omega,\mathbb R^n).$ 
\end{theorem}

\section{Proofs of the main results}
\setcounter{equation}{0}
\label{sec:3}
\begin{proof}[Proof of Theorem \ref{th:2.1}]
We will establish (i) first.\\ 
\textbf{Proof of (i).}

\noindent The arguments are borrowed from the classical Korn inequality theory. Following Kondratiev and Oleinik [\ref{bib:Kon.Ole.}] assume (\ref{2.1}) fails to hold. Hence there exists a sequence $\Bu_k\in W_0^{s,p}(\Omega,\mathbb R^n)$ such that 
\begin{equation}
\label{3.1}
\|\Bu_k\|_{W^{s,p}(\Omega)}=1\quad \text{and} \quad [\Bu_k]_{\mathcal{X}^{s,p}(\Omega)}\leq 1/k,\quad k=1,2,\dots
\end{equation}
By density we can assume without loss of generality that $\Bu_k\in C_c^1(\Omega,\mathbb R^n).$ From the compactness theorem 
[\ref{bib:DiN.Pat.Val.}, Theorem~7.1], we have that the sequence $\{\Bu_k\}$ is pre-compact in $L^p(\Omega),$ thus we can assume 
without loss of generality that 
\begin{equation}
\label{3.2}
\Bu_k\to \Bu \quad \text{in}\quad L^p(\Omega),
\end{equation}
for some field $\Bu\in L^p(\Omega).$ We have $\Bu_k-\Bu_m\in W_0^{s,p}(\Omega,\mathbb R^n),$ thus by (\ref{1.5}) ([\ref{bib:Men.Sco.2}, Thereom~1.1]) we have by the triangle inequality
\begin{align}
\label{3.3}
 \|\Bu_k-\Bu_m\|_{W^{s,p}(\Omega)}&\leq C([\Bu_k-\Bu_m]_{\mathcal{X}^{s,p}(\Omega)}+\|\Bu_k-\Bu_m\|_{L^p(\Omega)})\\ \nonumber
 & \leq  C([\Bu_k]_{\mathcal{X}^{s,p}(\Omega)}+[\Bu_m]_{\mathcal{X}^{s,p}(\Omega)}+\|\Bu_k-\Bu_m\|_{L^p(\Omega)})\\ \nonumber
 &\leq C(1/k+1/m+\|\Bu_k-\Bu_m\|_{L^p(\Omega)}).
 \end{align}
Note that conditions (\ref{3.2})-(\ref{3.3}) imply that the sequence $\{\Bu_k\}$ is Cauchy and thus is convergent in $W^{s,p}(\Omega).$ This gives 
\begin{equation}
\label{3.4}
\Bu_k\to \Bu \quad \text{in}\quad W^{s,p}(\Omega).
\end{equation}
From (\ref{3.4}) and the obvious estimate 
\begin{align*}
[\Bu]_{\mathcal{X}^{s,p}(\Omega)}&\leq [\Bu_k]_{\mathcal{X}^{s,p}(\Omega)}+[\Bu-\Bu_k]_{\mathcal{X}^{s,p}(\Omega)}\\
&\leq 1/k+[\Bu-\Bu_k]_{W^{s,p}(\Omega)}
\end{align*}
we also infer that $[\Bu]_{\mathcal{X}^{s,p}(\Omega)}=0,$ which implies that 
\begin{equation}
\label{3.5}
\Bu(\Bx)=\BA\cdot \Bx+\Bb,\quad\text{for a.e.}\quad\Bx\in\Omega,
\end{equation}
for some constant skew-symmetric matrix $\BA\in \mathbb R^{n\times n}$ and some vector $\Bb\in\mathbb R^n$ [\ref{bib:Tem.Mir.}, Proposition 1.2]. 
We aim to prove that $\BA=0$ and $\Bb=0.$ This can be achieved by either a utilization of a Hardy inequality or the notion of trace. 
The following Hardy inequality was proven by Dyda [\ref{bib:Dyda}, Theorem~1.1]; a stronger version was recently proven in 
[\ref{bib:Du.Men.Tia.}, Theorem~1.2].

\begin{theorem}
\label{th:3.1}
Let $s\in (0,1)$ and $p\in (1,\infty)$ such that $ps>1.$ Assume $\Omega\subset\mathbb R^n$ is an open bounded Lipschitz domain. 
Then there exists a constant $C=C(n,p,s,\Omega)$ such that 
\begin{equation}
\label{3.6}
\int_\Omega\frac{|\Bu(\Bx)|^p}{(\mathrm{dist}(\Bx,\partial\Omega))^{ps}} \leq C[\Bu]_{W^{s,p}(\Omega)}^p,
\end{equation}
for all vector fields $\Bu\in C_c^{1}(\Omega,\mathbb R^n).$ 
\end{theorem}
Note first that because of (\ref{3.2}) we have for any $\delta>0$ that 
$$\int_{\{\By\in \Omega \  :  \ \mathrm{dist}(\By,\partial\Omega)\geq \delta\}}\frac{|\Bu(\Bx)|^p}{\mathrm{dist}(\Bx,\partial\Omega)^{ps}}d\Bx
=\lim_{k\to\infty} \int_{\{\By\in \Omega \  :  \ \mathrm{dist}(\By,\partial\Omega)\geq \delta\}}\frac{|\Bu_k(\Bx)|^p}{\mathrm{dist}(\Bx,\partial\Omega)^{ps}}d\Bx,$$
hence we have due to Theorem~\ref{th:3.1} and the assumption (\ref{3.1}) the bound 
$$\int_{\{\By\in \Omega \  :  \ \mathrm{dist}(\By,\partial\Omega)\geq \delta\}}\frac{|\Bu(\Bx)|^p}{\mathrm{dist}(\Bx,\partial\Omega)^{ps}}d\Bx\leq C.$$
Consequently we discover letting $\delta$ go to zero:
\begin{equation}
\label{3.7}
\int_{\Omega}\frac{|\Bu(\Bx)|^p}{\mathrm{dist}(\Bx,\partial\Omega)^{ps}}d\Bx
=\int_{\Omega}\frac{|\BA\cdot\Bx+\Bb|^p}{\mathrm{dist}(\Bx,\partial\Omega)^{ps}}d\Bx\leq C.
\end{equation}
Note that as $ps>1$, then if $\BA\neq 0$ or $\Bb\neq 0,$ the integral in (\ref{3.7}) will necessarily diverge. Therefore we have $\BA=0$ and $\Bb=0$ and thus   $\Bu(\Bx)=0,$ for a.e. $\Bx\in\Omega,$ which contradicts the fact that $\Bu_k\to \Bu$ in $W^{s,p}(\Omega),$ while we have 
$\|\Bu_k\|_{W^{s,p}(\Omega)}=1$ by (\ref{3.1}) and (\ref{3.4}). 
This completes the proof of part (i).

\begin{remark}
Note that after (\ref{3.5}) one could proceed utilizing the trace theorem in [\ref{bib:Du.Men.Tia.}, Theorem~1.1]. Namely, Theorem~1.1 of [\ref{bib:Du.Men.Tia.}] states that in bounded Lipschitz domains $\Omega\subset\mathbb R^n,$ in the case $ps>1,$ any Sobolev function $\Bu\in W^{s,p}(\Omega)$ admits a trace 
$T\Bu\in W^{s-1/p,p}(\partial\Omega)$ on the boundary $\partial\Omega,$ that satisfies the fractional trace inequality 
$$\|T\Bu\|_{W^{s-1/p,p}(\partial\Omega)}\leq C\|\Bu\|_{W^{s,p}(\Omega)}.$$
Moreover, the trace coincides with the function itself for Lipschitz functions. Consequently, as $T\Bu_k=0$ and $T\Bu=\BA\cdot\Bx+\Bb,$ we have by (\ref{3.2}) and (\ref{3.4}) that
\begin{align*}
\|\BA\cdot\Bx+\Bb\|_{W^{s-1/p,p}(\partial\Omega)}&=\|T\Bu-T\Bu_k\|_{W^{s-1/p,p}(\partial\Omega)}\\
&\leq C\|\Bu-\Bu_k\|_{W^{s,p}(\Omega)}\to 0
\end{align*}
as $k\to\infty,$ which gives $\BA=0$ and $\Bb=0.$
\end{remark}

\begin{corollary}
\label{cor:3.3}
Due to Theorem~\ref{th:2.1} and Theorem~\ref{th:3.1}, one has in the case $ps>1$ the Korn-Hardy inequality 
\begin{equation}
\label{3.8}
\int_\Omega\frac{|\Bu(\Bx)|^p}{\mathrm{dist}(\Bx,\partial\Omega)^{ps}}d\Bx \leq C[\Bu]_{\mathcal{X}^{s,p}(\Omega)}^p,
\end{equation}
for all vector fields $\Bu\in W_0^{s,p}(\Omega,\mathbb R^n)$ by density. 
\end{corollary}

\noindent\textbf{Proof of (ii).}\\
For any $\delta>0$ denote $\Omega_{\delta}=\{\Bx\in\Omega \  :  \ \mathrm{dist}(\Bx,\Omega^c)>\delta\}.$ Let $\BA\in\mathbb R^{n\times n}$ be a constant nonzero skew-symmetric matrix. Let $\delta$ be fixed and so small that $\Omega_{5\delta}$ is not empty. For any $\epsilon\in (0,\delta)$ define the vector field 
\begin{equation}
\label{3.9}
\Bu(\Bx)=
\begin{cases}
\BA\cdot \Bx, & \Bx\in \Omega_{3\epsilon}\\[1ex]
0, & \Bx\in \Omega-\bar\Omega_{3\epsilon}.
\end{cases}
\end{equation}
Let $\eta_\epsilon(\Bx)\colon\mathbb R^n\to\mathbb R$ be the standard radial mollifier. Define the mollification  
$\Bu^\epsilon(\Bx)=(\Bu\ast \eta_\epsilon)(\Bx)\colon\Omega_\epsilon\to\mathbb R.$ It is clear that $\Bu(\Bx)=0$ for 
$\Bx\in\Omega_{\epsilon}-\bar \Omega_{2\epsilon},$ thus if we extend the field $\Bu^\epsilon$ into all of $\Omega$ by zero and denote the extended field again by 
$\Bu^\epsilon,$ it will preserve the smoothness property: $\Bu^\epsilon\in C_c^{\infty}(\Omega).$ The field $\Bu^\epsilon$ is basically an interpolation between 
$\BA\cdot\Bx$ and zero over the set $\Omega_\epsilon-\bar\Omega_{4\epsilon}.$ We claim that by choosing $\epsilon$ small enough, we can make both of the ratios 
$\frac{[\Bu]_{\mathcal{X}^{s,p}(\Omega)}}{[\Bu]_{W^{s,p}(\Omega)}}$ and $\frac{[\Bu]_{\mathcal{X}^{s,p}(\Omega)}}{\|\Bu\|_{L^p(\Omega)}}$ as small as we wish. Indeed, on one hand it is clear that as $\eta_\epsilon$ is radial, then for each $\Bx\in\Omega_{4\epsilon}$ we have 
$$\Bu^\epsilon(\Bx)=\int_{B_{\epsilon}(0)}(\BA\cdot(\Bx+\By)) \eta_\epsilon(\By)d\By=\BA\cdot \Bx+
\int_{B_{\epsilon}(0)}(\BA\cdot\By) \eta_\epsilon(\By)d\By=\BA\cdot\Bx,$$
hence as $\Omega_{4\delta}\subset\Omega_{4\epsilon},$ we have 
\begin{equation}
\label{3.10}
[\Bu^\epsilon]_{W^{s,p}(\Omega)}\geq [\BA\cdot \Bx]_{W^{s,p}(\Omega_{4\delta})}=C_1>0,
\quad 
\|\Bu^\epsilon\|_{L^p(\Omega)}\geq \|\BA\cdot\Bx\|_{L^p(\Omega_{4\delta})}=C_2>0,
\end{equation}
for some constants $C_1,C_2>0$ depending only on $\BA,$ $\delta$ and $\Omega.$ Let us now estimate the seminorm 
$[\Bu^\epsilon]_{\mathcal{X}^{s,p}(\Omega)}.$ In what follows within the ongoing case the constant $C>0$ 
will depend only on $p,s,n,\Omega$ and the matrix $\BA.$ We have that 
\begin{align}
\label{3.11}
[\Bu^\epsilon]_{\mathcal{X}^{s,p}(\Omega)}&
=\int_{\Omega_{4\epsilon}}\int_{\Omega_{4\epsilon}}
\frac{|(\Bu^\epsilon(\Bx)-\Bu^\epsilon(\By))\cdot(\Bx-\By)|^p}{|\Bx-\By|^{n+ps+p}}d\Bx d\By\\ \nonumber
&+2\int_{\Omega_{4\epsilon}}\int_{\Omega-\bar \Omega_{4\epsilon}}
\frac{|(\Bu^\epsilon(\Bx)-\Bu^\epsilon(\By))\cdot(\Bx-\By)|^p}{|\Bx-\By|^{n+ps+p}}d\Bx d\By\\ \nonumber
&+\int_{\Omega-\bar \Omega_{4\epsilon}}\int_{\Omega-\bar \Omega_{4\epsilon}}\frac{|(\Bu^\epsilon(\Bx)-\Bu^\epsilon(\By))\cdot(\Bx-\By)|^p}{|\Bx-\By|^{n+ps+p}}
d\Bx d\By\\ \nonumber
&=I_0+2I_1+I_2.
\end{align}
First let us observe that $\Bu^\epsilon(\Bx)=\BA\Bx$ in $\Omega_{4\epsilon}$, and since $\BA$ is skew-symmetric, we have
$\BA\Bx\cdot\Bx=0$ for an arbitrary $\Bx\in \mathbb R^n$. Thus $I_0=0$.
Further we have $|\Bu^\epsilon(\Bx)|\leq C$ for all $\Bx\in \Omega,$ thus we can estimate 
\begin{align}
\label{3.12}
I_1&=\int_{\Omega_{4\epsilon}}\int_{\Omega-\bar \Omega_{4\epsilon}}\frac{|(\Bu^\epsilon(\Bx)-\Bu^\epsilon(\By))\cdot(\Bx-\By)|^p}{|\Bx-\By|^{n+ps+p}}d\Bx d\By
\\ \nonumber
&\leq C\int_{\Omega_{4\epsilon}}\int_{\Omega-\bar \Omega_{4\epsilon}}\frac{1}{|\Bx-\By|^{n+ps}}d\Bx d\By.
\end{align}
For any fixed $\By\in\Omega-\bar \Omega_{4\epsilon}$ with $\mathrm{dist}(\By,\Omega_{4\epsilon})=t,$ where $0\leq t \leq4\epsilon,$ we clearly have 
$$\int_{\Omega_{4\epsilon}}\frac{1}{|\Bx-\By|^{n+ps}}d\Bx\leq \int_{|\Bz|\geq t}\frac{d\Bz}{|\Bz|^{n+ps}}=Ct^{-ps}.$$
For small enough $\epsilon>0,$ each of the level sets $\Gamma_t=\{\By\in\Omega-\bar \Omega_{4\epsilon}  \  : \  \mathrm{dist}(\By,\Omega_{4\epsilon})=t\}$ will have a perimeter smaller than $C=\mathrm{per}(\Omega)+1,$ thus we get integrating over the level sets $\Gamma_t:$
\begin{equation}
\label{3.13}
I_1\leq C\int_0^{4\epsilon}\int_{\Gamma_t} t^{-ps} dS dt \leq C\int_{0}^{4\epsilon} t^{-ps}dt=C\epsilon^{1-ps}.
\end{equation}
In order to estimate $I_2$ note first that $\Bu^\epsilon$ is a Lipschitz function with Lipschitz constant $C/\epsilon.$ This allows us to estimate 
\begin{align*}
I_2&=\int_{\Omega-\bar \Omega_{4\epsilon}}\int_{\Omega-\bar \Omega_{4\epsilon}}\frac{|(\Bu^\epsilon(\Bx)-\Bu^\epsilon(\By))\cdot(\Bx-\By)|^p}{|\Bx-\By|^{n+ps+p}}
d\Bx d\By\\ \nonumber
&=\int_{\Omega-\bar \Omega_{4\epsilon}}\int_{(\Omega-\bar \Omega_{4\epsilon})\cap\{|\Bx-\By|\leq \epsilon\}}
\frac{|(\Bu^\epsilon(\Bx)-\Bu^\epsilon(\By))\cdot(\Bx-\By)|^p}{|\Bx-\By|^{n+ps+p}}d\Bx d\By\\ \nonumber
&+\int_{\Omega-\bar \Omega_{4\epsilon}}\int_{(\Omega-\bar \Omega_{4\epsilon})\cap\{|\Bx-\By|>\epsilon\}}
\frac{|(\Bu^\epsilon(\Bx)-\Bu^\epsilon(\By))\cdot(\Bx-\By)|^p}{|\Bx-\By|^{n+ps+p}}d\Bx d\By\\ \nonumber
&\leq \frac{C}{\epsilon^p}\int_{\Omega-\bar \Omega_{4\epsilon}}\int_{(\Omega-\bar \Omega_{4\epsilon})\cap\{|\Bx-\By|\leq \epsilon\}}
\frac{1}{|\Bx-\By|^{n+ps-p}}d\Bx d\By\\ \nonumber
&+C\int_{\Omega-\bar \Omega_{4\epsilon}}\int_{(\Omega-\bar \Omega_{4\epsilon})\cap\{|\Bx-\By|>\epsilon\}}\frac{1}{|\Bx-\By|^{n+ps}}d\Bx d\By\\ \nonumber
&=C(I_2^1+I_2^2).
\end{align*}
For small enough $\epsilon,$ for the measure of $\Omega-\bar \Omega_{4\epsilon}$ we have 
$|\Omega-\bar \Omega_{4\epsilon}|\leq (1+\mathrm{per}(\Omega))\epsilon$ by Steiner's formula [\ref{bib:Schneider}], thus we can estimate
\begin{align}
\label{3.14}
I_2^1&=\frac{C}{\epsilon^p}\int_{\Omega-\bar \Omega_{4\epsilon}}\int_{(\Omega-\bar \Omega_{4\epsilon})\cap\{|\Bx-\By|\leq \epsilon\}}
\frac{1}{|\Bx-\By|^{n+ps-p}}d\Bx d\By\\ \nonumber
&\leq \frac{C}{\epsilon^p}\int_{\Omega-\bar \Omega_{4\epsilon}}\int_{|\Bz|\leq \epsilon}\frac{1}{|\Bz|^{n+ps-p}}d\Bz d\Bx\\ \nonumber
&=\frac{C}{\epsilon^p}\epsilon^{p-ps}|\Omega-\bar \Omega_{4\epsilon}|\\ \nonumber
&\leq C\epsilon^{1-ps}.
\end{align}
On the other hand for $I_2^2$ we have 
\begin{align}
\label{3.15}
I_2^2&=C\int_{\Omega-\bar \Omega_{4\epsilon}}\int_{(\Omega-\bar \Omega_{4\epsilon})\cap\{|\Bx-\By|>\epsilon\}}\frac{1}{|\Bx-\By|^{n+ps}}d\Bx d\By\\ \nonumber
&\leq \int_{\Omega-\bar \Omega_{4\epsilon}}\int_{|\Bz|>\epsilon}\frac{1}{|\Bz|^{n+ps}}d\Bz d\Bx\\ \nonumber
&=C\epsilon^{-ps}|\Omega-\bar \Omega_{4\epsilon}|\\ \nonumber
&\leq C\epsilon^{1-ps}.
\end{align}
Finally combining (\ref{3.11})-(\ref{3.15}) we discover
\begin{equation}
\label{3.16}
[\Bu^\epsilon]_{\mathcal{X}^{s,p}(\Omega)}\leq C\epsilon^{1-ps}.
\end{equation}
Putting together (\ref{3.10}) and (\ref{3.16}) we obtain due to the fact $ps<1,$
\begin{equation}
\label{3.17}
\frac{[\Bu^\epsilon]_{\mathcal{X}^{s,p}(\Omega)}}{[\Bu^\epsilon]_{W^{s,p}(\Omega)}}+
\frac{[\Bu^\epsilon]_{\mathcal{X}^{s,p}(\Omega)}}{\|\Bu^\epsilon\|_{L^p(\Omega)}}\leq C\epsilon^{1-ps}\to 0\quad \text{as}\quad \epsilon\to 0.
\end{equation}
This completes the proof of part (ii).
 
\end{proof}

\begin{proof}[Proof of Theorem \ref{th:2.2}] The proof is obtained as a direct consequence of Theorem~\ref{th:2.1} and Theorem~\ref{th:3.1}. Indeed, we have 
for any $\Bx\in\Omega$ and $\By\in \Omega^c$ the bound $|\Bx-\By|\geq \mathrm{dist}(\Bx,\partial\Omega)=d_\Bx,$ thus we can calculate
\begin{align}
\label{3.18}
[\bar\Bu]_{W^{s,p}(\mathbb R^n)}^p&=[\Bu]_{W^{s,p}(\Omega)}^p+2\int_{\Omega}\int_{\Omega^c}\frac{|\Bu(\Bx)|^p}{|\Bx-\By|^{n+ps}}d\Bx d\By\\ \nonumber
&\leq [\Bu]_{W^{s,p}(\Omega)}^p+2\int_{\Omega}\int_{(B_{d_\Bx}(\Bx))^c}\frac{|\Bu(\Bx)|^p}{|\Bx-\By|^{n+ps}}d\Bx d\By\\ \nonumber
&= [\Bu]_{W^{s,p}(\Omega)}^p+2\int_{\Omega}\frac{|\Bu(\Bx)|^p}{|d_{\Bx}|^{ps}}d\Bx,
\end{align}
consequently Theorem~\ref{th:3.1} implies 
\begin{equation}
\label{3.19}
[\bar\Bu]_{W^{s,p}(\mathbb R^n)}\leq C[\Bu]_{W^{s,p}(\Omega)}.
\end{equation}
Note that a similar calculation and Corollary~\ref{cor:3.3} imply 
\begin{align}
[\bar\Bu]_{\mathcal{X}^{s,p}(\mathbb R^n)}^p&=[\Bu]_{\mathcal{X}^{s,p}(\Omega)}^p
+2\int_{\Omega}\int_{\Omega^c}\frac{|\Bu(\Bx)\cdot(\Bx-\By)|^p}{|\Bx-\By|^{n+ps+p}}d\Bx d\By\\ \nonumber
&\leq [\Bu]_{\mathcal{X}^{s,p}(\Omega)}^p+2\int_{\Omega}\int_{(B_{d_\Bx}(\Bx))^c}\frac{|\Bu(\Bx)|^p}{|\Bx-\By|^{n+ps}}d\Bx d\By\\ \nonumber
&=[\Bu]_{\mathcal{X}^{s,p}(\Omega)}^p+2\int_{\Omega}\frac{|\Bu(\Bx)|^p}{|d_{\Bx}|^{ps}}d\Bx \\ \nonumber
&\leq C[\Bu]_{\mathcal{X}^{s,p}(\Omega)}^p.
 \end{align}

\end{proof}

\section*{Acknowledgements}
We thank the anonymous referee for useful comments and for pointing out some relevant literature that improved the presentation of the manuscript. 
The work of D.H. is supported by the National Science Foundation under Grants No. DMS-1814361.

\end{document}

%% file: mymacros.tex
\usepackage{amssymb,bm}
\usepackage{amsmath}
\usepackage{amsthm}
\usepackage{graphicx}
\usepackage[active]{srcltx} % Includes line number info into dvi file
\usepackage{hyperref}
\hypersetup{pdfborder=0 0 0}

\newtheorem{theorem}{{\sc Theorem}}[section]

\newtheorem{corollary}[theorem]{Corollary}
\newtheorem{remark}[theorem]{Remark}

\def\XXint#1#2#3{{\setbox0=\hbox{$#1{#2#3}{\int}$ }
\vcenter{\hbox{$#2#3$ }}\kern-.6\wd0}}

\bmdefine\BGa{\alpha}
\bmdefine\BGb{\beta}
\bmdefine\BGd{\delta}
\bmdefine\BGe{\epsilon}
\bmdefine\BGve{\varepsilon}
\bmdefine\BGf{\phi}
\bmdefine\BGvf{\varphi}
\bmdefine\BGg{\gamma}
\bmdefine\BGc{\chi}
\bmdefine\BGi{\iota}
\bmdefine\BGk{\kappa}
\bmdefine\BGl{\lambda}
\bmdefine\BGn{\eta}
\bmdefine\BGm{\mu}
\bmdefine\BGv{\nu}
\bmdefine\BGp{\pi}
\bmdefine\BGth{\theta}
\bmdefine\BGvth{\vartheta}
\bmdefine\BGr{\rho}
\bmdefine\BGvr{\varrho}
\bmdefine\BGs{\sigma}
\bmdefine\BGvs{\varsigma}
\bmdefine\BGt{\tau}
\bmdefine\BGj{\tau}
\bmdefine\BGu{\upsilon}
\bmdefine\BGo{\omega}
\bmdefine\BGx{\xi}
\bmdefine\BGy{\psi}
\bmdefine\BGz{\zeta}
\bmdefine\BGD{\Delta}
\bmdefine\BGF{\Phi}
\bmdefine\BGG{\Gamma}
\bmdefine\BGL{\Lambda}
\bmdefine\BGP{\Pi}
\bmdefine\BGT{\Theta}
\bmdefine\BGS{\Sigma}
\bmdefine\BGU{\Upsilon}
\bmdefine\BGO{\Omega}
\bmdefine\BGX{\Xi}
\bmdefine\BGY{\Psi}

\bmdefine\BCA{{\mathcal A}}
\bmdefine\BCB{{\mathcal B}}
\bmdefine\BCC{{\mathcal C}}
\bmdefine\BCD{{\mathcal D}}
\bmdefine\BCE{{\mathcal E}}
\bmdefine\BCF{{\mathcal F}}
\bmdefine\BCG{{\mathcal G}}
\bmdefine\BCH{{\mathcal H}}
\bmdefine\BCI{{\mathcal I}}
\bmdefine\BCJ{{\mathcal J}}
\bmdefine\BCK{{\mathcal K}}
\bmdefine\BCL{{\mathcal L}}
\bmdefine\BCM{{\mathcal M}}
\bmdefine\BCN{{\mathcal N}}
\bmdefine\BCO{{\mathcal O}}
\bmdefine\BCP{{\mathcal P}}
\bmdefine\BCQ{{\mathcal Q}}
\bmdefine\BCR{{\mathcal R}}
\bmdefine\BCS{{\mathcal S}}
\bmdefine\BCT{{\mathcal T}}
\bmdefine\BCU{{\mathcal U}}
\bmdefine\BCV{{\mathcal V}}
\bmdefine\BCW{{\mathcal W}}
\bmdefine\BCX{{\mathcal X}}
\bmdefine\BCY{{\mathcal Y}}
\bmdefine\BCZ{{\mathcal Z}}

\bmdefine\Bzr{ 0}
\bmdefine\Ba{ a}
\bmdefine\Bb{ b}
\bmdefine\Bc{ c}
\bmdefine\Bd{ d}
\bmdefine\Be{ e}
\bmdefine\Bf{ f}
\bmdefine\Bg{ g}
\bmdefine\Bh{ h}
\bmdefine\Bi{ i}
\bmdefine\Bj{ j}
\bmdefine\Bk{ k}
\bmdefine\Bl{ l}
\bmdefine\Bm{ m}
\bmdefine\Bn{ n}
\bmdefine\Bo{ o}
\bmdefine\Bp{ p}
\bmdefine\Bq{ q}
\bmdefine\Br{ r}
\bmdefine\Bs{ s}
\bmdefine\Bt{ t}
\bmdefine\Bu{ u}
\bmdefine\Bv{ v}
\bmdefine\Bw{ w}
\bmdefine\Bx{ x}
\bmdefine\By{ y}
\bmdefine\Bz{ z}
\bmdefine\BA{ A}
\bmdefine\BB{ B}
\bmdefine\BC{ C}
\bmdefine\BD{ D}
\bmdefine\BE{ E}
\bmdefine\BF{ F}
\bmdefine\BG{ G}
\bmdefine\BH{ H}
\bmdefine\BI{ I}
\bmdefine\BJ{ J}
\bmdefine\BK{ K}
\bmdefine\BL{ L}
\bmdefine\BM{ M}
\bmdefine\BN{ N}
\bmdefine\BO{ O}
\bmdefine\BP{ P}
\bmdefine\BQ{ Q}
\bmdefine\BR{ R}
\bmdefine\BS{ S}
\bmdefine\BT{ T}
\bmdefine\BU{ U}
\bmdefine\BV{ V}
\bmdefine\BW{ W}
\bmdefine\BX{ X}
\bmdefine\BY{ Y}
\bmdefine\BZ{ Z}